\documentclass[a4paper,12pt]{article}
\usepackage[utf8]{inputenc}
\usepackage{amssymb, amsmath}
\usepackage{amsthm}
\newtheorem*{lemma}{Statement}

\newtheorem*{definition}{Definition}
\newtheorem*{remark}{Remark}
\newtheorem{corollary}{Corollary}
\newtheorem*{application}{Application}

\newtheorem{theoreme}{Theorem}
\newtheorem*{conjecture}{Conjecture}

\newtheorem*{lemma1}{Lemma}
\title{Some Results on Zero-Sum Sequences in $Z_{p}^{3}$}
\author{\bf {Satwik Mukherjee}\\
University of Lille1\\
 Email  : satwiklille1@gmail.com}
\date{September 9, 2013}

\begin{document}

\maketitle

\begin{abstract}
 Kemnitz Conjecture [9] states that if we take a sequence of elements in $Z_{p}^{2}$ of length $4p-3$,
 $p$ is a prime number, then it has a subsequence of length $p$, whose sum is $0$ modulo $p$. It is
 known that in $Z_{p}^{3}$ to get a similar result we have to take a sequence of length atleast $9p-8$
 . In this paper we will show that if we add a condition on the chosen sequence, then we can get a good upper and a lower bound for
 which  similar results hold. 
\end{abstract}

\section*{Introduction}

 Denoting by $s_{k}(Z_{n}^{d})$  the smallest integer $t$ such that any set of $t$ lattice-points in the $d$-dimensional
 Euclidean space contains a subset of cardinality $kn$, the sum of whose elements is divisible by $n$,
 it was first proved by Erdős, Ginzburg and Ziv   [4], that $s_{1}(Z_{n})=2n-1$. 
 Kemnitz' Conjecture $s_{1}(Z_{n}^{2})=4n-3$ was open for about twenty years and was proved by Reiher in [9] after 
 a series of results by Gao [5], Rónyai [10] and others. Up to now the best general bounds for odd primes $p$ and higher dimensions $d$ are
 $s_{1}(Z_{p}^{d})\geqslant 1.125^{\lfloor \frac{d}{3}\rfloor}2^{d}(p-1)+1$,
by Elsholtz [3],  and $ s_{1}(Z_{p}^{d})\leqslant (cd\log d)^{d}p$ by Alon and Dubiner [1], where $c$ is a constant. 
They conjectured that $s_{1}(Z_{p}^{d})\leqslant c^{d}p$. In $2001$, Elsholtz [3] showed that $s_{1}(Z_{n}^{3})\geqslant 9n-8$. Bhowmik and Schalge-Puchta [2]
proved that $s_{1}(Z_{p}^{3})=9p-8$ for $p\xrightarrow{}\infty$, $p$ is a prime number. 
Hence, it is natural to ask whether $s_{1}(Z_{p}^{3})=9p-8$ for all $p$. We are as yet unable to answer this question. However, we study the constant $s_{I}(Z_{p}^{3})$ for 
certain sequences $I$. Kubertin [8], Gao, Thangadurai [6] and Geroldinger, Grynkiewicz, Schmid [7] have studied some properties
of these kind of constants. Gao, Thangadurai [6] studied this constant for groups $G\cong Z_{n}^{d}$
when $d$=3 or 4 and proved that $s_{k}(Z_{p}^{3})=kp+3p-3$ for every $k\geqslant 4$, where $p$
is a prime number. Kubertin [8] further extended this result by proving that $s_{k}(Z_{q}^{3})=(k+3)q-3$ 
for $k\geqslant 3$ and $q$ be a prime power of $p>3$ and 
$s_{k}(Z_{q}^{4})=(k+4)q-4$ for $k\geqslant 4$ and $p\geqslant 7$, $p$ is a prime number and $q$
is a prime power of $p$. She conjectured that for positive integers $k\geqslant d$ and $n$
we have $s_{k}(Z_{n}^{d})=(k+d)n-d$ and proved that the conjecture holds for $s_{np}(Z_{q}^{d})$ where
$p$ is a prime number and $q$ is a power of $p$. Geroldinger, Griynkiewicz, Schmid [7] defined 
for a finite abelian group $G$ and a positive integer $d$, $s_{d\mathbb{N}}(G)$ to be the smallest 
integer $l\in \mathbb{N}_{0}$ such that every sequence $S$ over $G$ of length $|S|\geqslant l$ has a non-empty
zero-sum subsequence $T$ of length $|T|\equiv 0$ mod $d$. They showed that, 
$\\$Let $d\in \mathbb{N}$ and let $n=exp(G)$. Suppose $G$ is cyclic. Then $s_{d\mathbb{N}}(G)=lcm(n,d)+gcd(n,d)-1$.

They also determined $s_{d\mathbb{N}}(G)$ for all $d\geqslant1$ when $G$ has rank atmost two and, under mild
conditions on $d$, and obtained precise values in the case of $p-groups$. Continuing on this line in this paper we give an upper bound and a lower bound for $s_{1}(Z_{n}^{3})$ of a particular
kind of sequences. We have used the idea of `lifting of an equation' (Explained Later) by Reiher [9] for studying
 some properties of the sequences in $Z_p^3$ and have generalized the function used in Ronyai's Method [8]
  to prove one of our theorems. We have used the `Polynomial Methods' to study the zero-sum properties 
 of the sequences in $Z_p^3$. We must note that $s_{1}(Z_{n}^{d})$ is a completely multiplicative function of $n$. Here,
 we will prove our results for the prime numbers $p$ which essentially proves for the other
 integers n also. Throughout the text $p$ denotes a prime number and n stands for any integer.

\section*{Main Results and Two Applications of Polynomial Methods}
  
\subsection*{Main Theorems}
If $I$ is a sequence of elements in $Z_{p}^{3}$, then $N^{kp}(I)$ denotes the number of subsequences of $I$ of length $kp$,
whose sum is $0$ modulo $p$. By $a\equiv b$, we mean $a\equiv b$ mod $p$. In this paper , we will prove the following three theorems :

\begin{theoreme}
Let  $J$ be a sequence of elements in $Z_{p}^{3}$ with $|J|=7p-3$. Let, for all $ I \subset J$ with $|I|=4p-3$, $N^{2p}(I)\equiv c$ mod p, where c is a fixed number, then $J$ has a subsequence  of length $p$,
whose sum is 0 modulo p when $p>7$.
\end{theoreme}

\begin{theoreme}
 There is a sequence of elements $J$ in $Z_{p}^{3}$ with $|J|=4p-4+\frac{p-1}{2}$ such that for all $ I \subset J$ with $|I|=4p-3$, $N^{2p}(I)\equiv 0$ mod p,
 and $J$ does not have any subsequence  of length $p$ , whose sum is 0 modulo p.
\end{theoreme}

  We also have investigated the case when $N^{p}(J)=0$, $J$ is a sequence of  $Z_{p}^{3}$. For such a sequence $J$ with $|J|=9p-3$,
 we checked whether it has  the subsequences  of length $ip$,
whose sum is 0 modulo p, for  $2\leqslant i\leqslant 8$. And we 
 have come up with the result :
 \begin{theoreme}Let $J$ = $ \{(a_{i}, b_{i}, c_{i}) $ , $1\leqslant i \leqslant (9p-3)\}$ be a sequence in $Z_{p}^{3}$.
 We have either $N^{p}(J)>0$ or six of the $N^{ip}(J)$'s are not congruent to 0, $2\leqslant i\leqslant 8$.
\end{theoreme}

We need the following definitions to prove the above theorem.

\begin{definition}We define g(x) = $ x_{1}^{p-1} +...+ x_{9p-3}^{p-1}$. And with the help of it we  define,

 $S_{i} = (\binom {g(x)} {p} - i)$, $Q_{1} = \binom{ a_{1}x_{1}^{p-1}+...+a_{9p-3}x_{9p-3}^{p-1}-1} {p-1}$, 
 
 $\\$
 
 $Q_{2} = \binom{ b_{1}x_{1}^{p-1}+...+b_{9p-3}x_{9p-3}^{p-1}-1} {p-1}$, $Q_{3} = \binom{ c_{1}x_{1}^{p-1}+...+c_{9p-3}x_{9p-3}^{p-1}-1} {p-1}$.
 
$\\$ 
 We also define a new function  $P_{18}(x) = \binom{g(x)-1} {p-1} Q_{1}Q_{2}Q_{3} S_{2}S_{3}...S_{7}$   where 
 only $S_{1}, S_{8}$ are missing. $P_{ij}(x)$ is defined similarly in which only $S_{i},S_{j}$ are missing.
 
 \end{definition}

  \begin{proof}[\bf Remark] ``By lifting one of the above equations'' follows the same technique that Reiher [9]
  used in his paper.
 
 As an example, if $J$ is a sequence in $Z_{p}^{2}$ of cardinality $3p-3$, then by 
 the polynomial method [9] we get the equation $1-N^{p-1}(J)-N^{p}(J)+N^{2p-1}(J)+N^{2p}(J)\equiv0$.
 Now, let $X$ be a sequence of $Z_{p}^{2}$ of cardinality $4p-3$. Then, we get:- 
 $\Sigma \{1-N^{p-1}(J)-N^{p}(J)+N^{2p-1}(J)+N^{2p}(J)\}\equiv0$, where the sum is extended over all
 $J\subset X$ of cardinality $3p-3$. Analysing the number of times each subsequence is counted we 
 obtain $\binom{4p-3} {3p-3}- \binom{3p-2}{2p-2}N^{p-1}(X)-\binom{3p-3}{2p-3}N^{p}(X)+\binom{2p-2}{p-2}N^{2p-1}(X)+\binom{2p-3}{p-3}N^{2p}(X)\equiv0$
 . After the reduction we get, $3- 2N^{p-1}(X)-2N^{p}(X)+N^{2p-1}(X)+N^{2p}(X)\equiv0$. This equation
 will be called a lifting of $1-N^{p-1}(J)-N^{p}(J)+N^{2p-1}(J)+N^{2p}(J)\equiv0$.
 \end{proof}
\subsection*{Existence of a $2p$-Zero Sum Sequence in $Z_p^3$}
\begin{lemma}  If I is a sequence of elements in $Z_{p}^{3}$ and $|I|=6p-3$, then  $I$ has a subsequence of 
length $2p$, whose sum is congruent to 0 modulo p.
\end{lemma}

\begin{proof} If $N^{4p}(I)>0$ then there exists a susequence $J$ of $I$ such that $|J|=4p$ and $\sum J \equiv 0$.
 Define $K=J-{x}$ where $x$ is any element of $J$.
 Using the technique used by Reiher [9] we get the equation, $1-N^{p}(K)+N^{2p}(K)-N^{3p}(K)\equiv 0$.
 So, one of the  $N^{p}(K),N^{2p}(K),N^{3p}(K)$ has to be non-zero.
 And we  get that:- either $J$ has a subsequence of length $p$ whose sum is $0$ mod $p$ or it has
 a subsequence of length $2p$ whose sum is $0$ mod $p$. Now let  $|J|=5p-3$.Then we have this equation  $4-3N^{p}(J)+2N^{2p}(J)-N^{3p}(J)\equiv 0$ by lifting the 
equation [9]   $1-N^{p}(I)+N^{2p}(I)-N^{3p}(I)\equiv 0$ where $|I|=4p-3$. We also have the equation $1-N^{p}(J)+N^{2p}(J)-N^{3p}(J)+N^{4p}(J)\equiv 0$.
Our claim is that either $J$ has a subsequence of length $p$ whose sum is $0$ mod $p$ or it has
a subsequence of length $2p$ whose sum is $0$ mod $p$. If not, then we get $N^{3p}(J) \equiv 4$ and therefore we obtain $N^{4p}(J)\equiv 3$ and hence either
$J$ has a subsequence of length $p$ whose sum is $0$ mod $p$ or it has
 a subsequence of length $2p$ whose sum is $0$ mod $p$.
$\\$Finally, if $|L|=6p-3$,then $6p-3>5p-3$ , so we have either  $L$ has a subsequence of length $p$ whose sum is $0$ mod $p$ or it has
 a subsequence of length $2p$ whose sum is $0$ mod $p$. Let $N^{2p}(L)=0$.
 So, we have  $T \subset L$ such that $|T|=p$ and $\sum T \equiv 0$. Now  $|L-T|=5p-3$ and by repeating the previous argument we will get another $p$ sequence whose
sum is $0$ mod $p$, call it $V$. Then $V \bigcup T$ is a $2p$ sequence whose sum is $0$.
\end{proof}

\subsection*{Existence of a $3p$-Zero Sum Sequence in $Z_p^3$}
\begin{lemma} 
If $J$ is a sequence of $Z_{p}^{3}$ such that $|J|=7p-3$, then  $J$ has a subsequence of 
length $3p$, whose sum is congruent to 0 modulo p .
\end{lemma}

\begin{proof}  Kubertin [8] proved that this bound can be reduced to $6p-3$. But, by only using 
the polynomial method [9] we get an upper bound equals to $7p-3$.

Firstly, we show that if $|I|=5p$ and $\sum I \equiv 0$ then either $I$ has a subsequence of length $p$ whose sum is $0$ mod $p$ or it has
 a subsequence of length $2p$ whose sum is $0$ mod $p$.  Let,  $H=I-{x}$ where $x$ is an element of $I$,then we have the equation:  $1-N^{p}(H)+N^{2p}(H)-N^{3p}(H)+N^{4p}(H)\equiv 0$.
 So, if  $N^{p}(I)=0,N^{2p}(I)=0$ then either  $N^{3p}(H)>0$ or  $N^{4p}(H)>0$ and which contradicts the assumption. Now, we show that if $|J|=6p-3$ then either $J$ has a subsequence of length $p$ whose sum is $0$ mod $p$ or it has
 a subsequence of length $2p$ whose sum is $0$ mod $p$. If not, then we have these two equations,

 $1-N^{3p}(J)+N^{4p}(J)-N^{5p}(J)\equiv 0.$

and  $5-2N^{3p}(J)+N^{4p}(J)\equiv 0$ by lifting the equation : $1-N^{p}(L)+N^{2p}(L)-N^{3p}(L)+N^{4p}(L)\equiv 0$
 where  $|L|=5p-3$. And we get the equation $-4+N^{3p}(J)-N^{5p}(J) \equiv 0$. So, either $N^{3p}(J)>0$ or $N^{5p}(J)>0$ and hence the result follows.
  Now, if $|J|=7p-3$ and $J$ has either a subsequence of length $p$ whose sum is $0$ mod $p$ or it has
 a subsequence of length $3p$ whose sum is $0$ mod $p$ then it is okay. Otherwise, we will have these equations :
 \begin{enumerate}
 \item $6+4N^{2p}(J)+2N^{4p}(J)-N^{5p}(J)\equiv 0$ by lifting the equation : $1-N^{p}(H)+N^{2p}(H)-N^{3p}(H)+N^{4p}(H)-N^{5p}(H)\equiv 0$ where $|H|=6p-3$.

 \item  $15+6N^{2p}(J)+N^{4p}(J)\equiv 0$  by lifting  : $1-N^{p}(H)+N^{2p}(H)-N^{3p}(H)+N^{4p}(H)\equiv 0$ where $|H|=5p-3$.
 
 \item  $5+N^{2p}(J)\equiv 0$  by lifting : $1-N^{p}(H)+N^{2p}(H)-N^{3p}(H)\equiv 0$  where $|H|=4p-3$.\end{enumerate}

  And all these gives us the new equation  $N^{5p}(J)\equiv 16$ .  And  we have seen that $|T|=5p$, $\sum T\equiv 0$  implies 
 either $T$ has a subsequence of length $p$ whose sum is $0$ mod $p$ or it has
 a subsequence of length $2p$ whose sum is $0$ mod $p$. But as we have assumed that $N^{p}(T)=0$, therefore there exists $M \subset T$
 such that $|M|=2p$ and $\sum M \equiv 0$ . And $T-M$ satisfies our condition. 
 \end{proof}

\begin{corollary} It is also clear that if $I$ is a sequence of $Z_{p}^{3}$ with $|I|=8p-3$ where $p$ is a prime number,
then $I$ has a subsequence of length $4p$ whose sum is $0$ mod p and if $|I|=9p-3$ then $I$ has subsequence
of length $5p$ whose sum is $0$ mod p.
 \end{corollary}

 \section*{Proofs }
 \subsection*{Results On $Z_p^3$ Sequences leading to the proofs}
  
  We have used this result in several ocassions without mentioning it, $\binom{kp-c}{rp-c}\equiv \frac{(k-1)...r}{(k-r)!}$
  and we are assuming that $p>7$ for the rest of the text.
 
 $\\$
 Let $J$ be a sequence of $Z_{p}^3$ with $|J|=9p-3$. Then we have the following equations by using the technique used in [9] :
 
 \begin{enumerate}
 
\item  $1-N^{p}(J)+N^{2p}(J)-N^{3p}(J)+N^{4p}(J)-N^{5p}(J)+N^{6p}(J)-N^{7p}(J)+N^{8p}(J) \equiv 0.$

\item $1-N^{p}(J_{1})+N^{2p}(J_{1})-N^{3p}(J_{1})+N^{4p}(J_{1})-N^{5p}(J_{1})+N^{6p}(J_{1})-N^{7p}(J_{1}) \equiv 0$ where $j_{1}$ is a subsequence of $J$ with $|J_{1}|=8p-3.$

\item $1-N^{p}(J_{2})+N^{2p}(J_{2})-N^{3p}(J_{2})+N^{4p}(J_{2})-N^{5p}(J_{2})+N^{6p}(J_{2}) \equiv 0$ where $J_{2}$ is a subsequence of $J$ wit $|J_{2}|=7p-3.$

 \item $1-N^{p}(J_{3})+N^{2p}(J_{3})-N^{3p}(J_{3})+N^{4p}(J_{3})-N^{5p}(J_{3}) \equiv 0$ where $J_{3}$ is a subsequence of $J$ with $|J_{3}|=6p-3.$

 \item $1-N^{p}(J_{4})+N^{2p}(J_{4})-N^{3p}(J_{4})+N^{4p}(J_{4}) \equiv 0$ where $J_{4}$ is a subsequence of $J$ with $|J_{4}|=5p-3.$

 \item $1-N^{p}(J_{5})+N^{2p}(J_{5})-N^{3p}(J_{5}) \equiv 0$ where $J_{5}$ is a subsequence of $J$ with $|J_{5}|=4p-3.$
 
\end{enumerate}

 By lifting the above equations we get :
 
 \begin{enumerate}
 
 \item $56-21N^{p}(J)+6N^{2p}(J)-N^{3p}(J) \equiv 0$ after lifting the above equation $(6)$.

\item $70-35N^{p}(J)+15N^{2p}(J)-5N^{3p}(J)+N^{4p}(J) \equiv 0$ after lifting the above equation $(5)$.

\item $56-35N^{p}(J)+20N^{2p}(J)-10N^{3p}(J)+4N^{4p}(J)-N^{5p}(J) \equiv 0$ after lifting the above equation $(4)$.

\item $28-21N^{p}(J)+15N^{2p}(J)-10N^{3p}(J)+6N^{4p}(J)-3N^{5p}(J)+N^{6p}(J) \equiv 0$ after lifting the above equation $(3)$.

\item $8-7N^{p}(J)+6N^{2p}(J)-5N^{3p}(J)+4N^{4p}(J)-3N^{5p}(J)+2N^{6p}(J)-N^{7p}(J) \equiv 0$ after lifting the above equation $(2)$.

\item $1-N^{p}(J)+N^{2p}(J)-N^{3p}(J)+N^{4p}(J)-N^{5p}(J)+N^{6p}(J)-N^{7p}(J)+N^{8p}(J) \equiv 0$. 
 
\end{enumerate}
 
  Now, assume that, $N^{2p}(J)\equiv k$  and $N^{p}(J)=0$.  Then, we can rewrite the above equations in the form of a matrix

 $\\$

$   \begin{pmatrix}
    1 & 0 & 0 & 0 & 0 & 0 & 0 \\
    6 & -1 & 0 & 0 & 0 & 0 & 0 \\
    15 & -5 & 1 & 0 & 0 & 0 & 0 \\
    20 & -10 & 4 & -1 & 0 & 0 & 0 \\
    15 & -10 & 6 & -3 & 1 & 0 & 0 \\
    6 & -5 & 4 & -3 & 2 & -1 & 0 \\
    1 & -1 & 1 & -1 & 1 & -1 & 1 \\
   
   \end{pmatrix} $ $\times$ $\begin{pmatrix}
   N^{2p}(J) \\
   N^{3p}(J) \\
   N^{4p}(J) \\
   N^{5p}(J) \\
   N^{6p}(J) \\
   N^{7p}(J) \\
   N^{8p}(J) \\
   
   \end{pmatrix} $ = $ \begin{pmatrix}
   k \\
   -56 \\
   -70 \\
   -56 \\
   -28 \\
   -8 \\
   -1 \\
   
   \end{pmatrix} $

   $\\$ 
   
   And we get,

   $\begin{pmatrix}
   N^{2p}(J) \\
   N^{3p}(J) \\
   N^{4p}(J) \\
   N^{5p}(J) \\
   N^{6p}(J) \\
   N^{7p}(J) \\
   N^{8p}(J) \\
   
   \end{pmatrix} $ = $   \begin{pmatrix}
    1 & 0 & 0 & 0 & 0 & 0 & 0 \\
    6 & -1 & 0 & 0 & 0 & 0 & 0 \\
    15 & -5 & 1 & 0 & 0 & 0 & 0 \\
    20 & -10 & 4 & -1 & 0 & 0 & 0 \\
    15 & -10 & 6 & -3 & 1 & 0 & 0 \\
    6 & -5 & 4 & -3 & 2 & -1 & 0 \\
    1 & -1 & 1 & -1 & 1 & -1 & 1 \\
   
   \end{pmatrix} $ $\times$ $ \begin{pmatrix}
   k \\
   -56 \\
   -70 \\
   -56 \\
   -28 \\
   -8 \\
   -1 \\
   
   \end{pmatrix} $ = $ \begin{pmatrix}
   k \\
   6k+56 \\
   15k+210 \\
   20k+336 \\
   15k+252 \\
   6k+120 \\
   k+21\\
   
   \end{pmatrix} $
   
   $\\$
   
   Among the numbers $ k,
   6k+56, 
   15k+210, 
   6k+120 , 
   k+21 , 20k+336, 15k+252$,  $p$ cannot divide two of these numbers simultaneously, except $20k+336, 15k+252$  as $p>7$. 
   (i.e. if $p$ divides $6k+56$ and $6k+120$, then $p$ divides $64$ but $p$ is a prime number $>7$.
   So, it is not possible. The similar argument proves the claim.)

   \begin{corollary}

   If $J$ is a sequence in $Z_{p}^{3}$ with $|J|=8p-3$, $N^{2p}(J)\equiv l$ mod p and $N^{p}(J)=0$, then we get 
   
   $\\$

   $\begin{pmatrix}
   N^{2p}(J) \\
   N^{3p}(J) \\
   N^{4p}(J) \\
   N^{5p}(J) \\
   N^{6p}(J) \\
   N^{7p}(J) \\

   \end{pmatrix} $ = $ \begin{pmatrix}
   l \\
   5l+35 \\
   10l+105 \\
   10l+122 \\
   5l+52 \\
   l+1 \\
      \end{pmatrix} $ and if $p$ divides one of the numbers on the right hand side then it cannot
      divide the others if $p \neq 13, 17, 19, 47, 61$.
      \end{corollary}

   \begin{corollary}
      
       If $J$ is a sequence in $Z_{p}^{3}$ with $|J|=7p-3$, $N^{2p}(J)\equiv m$ mod p and $N^{p}(J)=0$, then we get 
      
      $\\$
   
   $\begin{pmatrix}
   N^{2p}(J) \\
   N^{3p}(J) \\
   N^{4p}(J) \\
   N^{5p}(J) \\
   N^{6p}(J) \\

   \end{pmatrix} $ = $ \begin{pmatrix}
   m \\
   4m+20 \\
   6m+45 \\
   4m+36 \\
   m+10 \\
   
      \end{pmatrix} $ 
and if $p$ divides one of the numbers on the right hand side then it cannot
      divide the others.

      \end{corollary}
\begin{corollary}

      If $J$ is a sequence in $Z_{p}^{3}$ with $|J|=6p-3$, $N^{2p}(J)\equiv t$ mod p and $N^{p}(J)=0$, then we get 
   $\\$
   
   $\begin{pmatrix}
   N^{2p}(J) \\
   N^{3p}(J) \\
   N^{4p}(J) \\
   N^{5p}(J) \\

   \end{pmatrix} $ = $ \begin{pmatrix}
   t \\
   3t+10 \\
   3t+15 \\
   t+6 \\
    
      \end{pmatrix} $ and if $p$ divides one of the numbers on the right hand side then it cannot
      divide the others.
      
       \end{corollary}
      
\begin{corollary}

     If $J$ is a sequence in $Z_{p}^{3}$ with $|J|=5p-3$, $N^{2p}(J)\equiv r$ mod p and $N^{p}(J)=0$ then we get 
   $\\$
   
   $\begin{pmatrix}
   N^{2p}(J) \\
   N^{3p}(J) \\
   N^{4p}(J) \\

   \end{pmatrix} $ = $ \begin{pmatrix}
   r \\
   2r+4 \\
   r+3 \\

      \end{pmatrix} $   
and if $p$ divides one of the numbers on the right hand side then it cannot
      divide the others.
      \end{corollary}

      \begin{remark}
      
      If $|I|=9p-3$, $N^{p}(J)=0$ and $N^{2p}(J)\equiv c$, $c$ is a fixed  number $\forall J$ such that $|J|=4p-3$,$J \subset I$ then
 we have the following relations $5r\equiv3t, 5r\equiv2m, 7r\equiv2l, 3k\equiv14t$ and $3t\equiv 10c,2m\equiv 10c,r\equiv 2c$. Here,
 $t, r, m, l, k$ are the variables defined earlier.
      \end{remark}

   \begin{application}
   
   $\\$  If $J$ is a sequence in $Z_{p}^{3}$ with $|J|=9p-3$ and $N^{p}(J)=0$ , then $J$ has a 
   subsequence of length $6p$ whose sum is $0$ mod $p$ where $p$ is a prime number $>7$ and $p \neq 13, 17, 19, 47, 61$ .
   \end{application}

\begin{proof}

    If $N^{6p}(J)=0$, then by using the earlier calculations we get $N^{8p}(J)\neq 0$. So, there exists $I \subset J$
   such that $|I|=8p$ and $\sum I \equiv 0$. Let M = $I-\{x, y, z\}$. Here, $|M|=8p-3$. So, using the 
   corollary1 we get that $N^{7p}(M)\neq 0$. So, there exist $L\subset M$ such that $|L|=7p$ and 
   $\sum L \equiv 0$. Now, $|I-L|=p$ and $\sum (I-L) \equiv 0$. But, it is a contradiction to our hypothesis.
    So, $N^{6p}(J) \neq 0$.
\end{proof}   
   
 \subsection*{Proof Of The Main Results}

\begin{proof}[\bf Proof of Theorem1] Let there be a subsequence $I \subset J$, $|I|=6p$ and $\Sigma I\equiv 0$.
 Now, if there exists a subsequence $K\subset I$ such that $|K|=5p$ and $\Sigma K\equiv 0$ then we are done.
 Otherwise, $t\equiv -6 \Rightarrow  3t\equiv -18\equiv 10c \Rightarrow 5c\equiv -9$, where $t = N^{2p}(I^{`})$,$|I^{`}| = 6p-3$
and $I^{`}\subset I$. Here,$N^{4p}(I^{`})\equiv -3$. Let $I^{``}\subset I^{`}, |I^{``}|=4p, \Sigma I^{``}\equiv 0$.
 Let $L\subset I^{``}, |L|=4p-3$. We get the equation, 
$1+N^{2p}(L)-N^{3p}(L)\equiv 0 \Rightarrow 5+5c-5N^{3p}(L)\equiv 0 \Rightarrow N^{3p}(L) \neq 0$ as $5c\equiv -9$. So, we have got a $p$ zero-sequence.
If we do not have a $6p$ zero sequence inside $J$, then $m\equiv -10$ ie $N^{5p}(J)\neq 0$ and $10c\equiv -20$,
 $r\equiv -4$, so we are done.
\end{proof}
$\\$
\begin{proof}[\bf Proof of Theorem2] Consider the set $I=\{(0,0,0)^{p-1}, (1,0,0)^{p-1}, (0,1,0)^{p-1}, 
$\\$(0,0,1)^{p-1}, (1,1,1)^{\frac{p-1}{2}}\}$. Then, $|I|=4p-4+\frac{p-1}{2}$. $I$  neither has a $p$-zero-sum sequence of elements nor has a $2p$-zero-sum sequence of elements. So, $I$ satisfies the condition that
 $N^{2p}(I)\equiv 0$ but it does not have a $p$-zero-sum sequence.
\end{proof}
 \begin{remark}So, $4p-4+\frac{p-1}{2}$ is a lower bound for this result.\end{remark}

\begin{lemma1}
 $ \Sigma_{x \in Z_{p}^{9p-3}} P_{mn}(x) \equiv \Sigma_{x\in Z_{p}^{9p-3}} P_{rs}(x)$ for $m \neq n$ , $r\neq s$.
\end{lemma1}

\begin{proof}We will prove that $ \Sigma_{x \in Z_{p}^{9p-3}} P_{18}(x) \equiv $ $\Sigma_{x \in Z_{p}^{9p-3}} P_{17}(x)$.
And the rest can be derived from this : 

$\\$
$\Sigma_{x \in Z_{p}^{9p-3}} P_{18}(x)$ = $\Sigma_{x \in Z_{p}^{9p-3}}  \binom{g(x)-1} {p-1} Q_{1}Q_{2}Q_{3} S_{2}S_{3}...S_{7}$ 
$\\$= $\Sigma_{x \in Z_{p}^{9p-3}}  \binom{g(x)-1} {p-1} Q_{1}Q_{2}Q_{3} S_{2}...S_{6} (\binom {g(x)} {p} - 7)$
$\\$
= $ \Sigma_{x \in Z_{p}^{9p-3}} \binom{g(x)-1} {p-1} Q_{1}Q_{2}Q_{3} S_{2}S_{3}...S_{6} (\binom {g(x)} {p} - 8+1)$ 
$\\$= $\Sigma_{x \in Z_{p}^{9p-3}} P_{17}(x) +  \Sigma_{x \in Z_{p}^{9p-3}} \binom{g(x)-1} {p-1} Q_{1}Q_{2}Q_{3} S_{2}S_{3}...S_{6}$
$\\$
= $ \Sigma_{x \in Z_{p}^{9p-3}} P_{17}(x) +  \Sigma_{x \in Z_{p}^{9p-3}} R$.

$\\$
We will calculate the value of $R$ directly.

$\\$
$R = \Sigma_{1 \leqslant i \leqslant (9p-3)} a_{k_{1}...k_{9p-3}} \Pi x_{i}^{k_{i}(p-1)}$, where
$k_{i} \geqslant 0$.

$\\$
And $\Sigma_{x \in Z_{p}^{9p-3}} a_{k_{1}...k_{9p-3}} \Pi x_{i}^{k_{i}(p-1)} \equiv 0$ if one of 
the $k_{i}=0$. Hence, $R$ can be non-zero if $degree(R)$ is at least equal to $(9p-3)(p-1)$. But
in our case $degree(R)$ is $(9p-4)(p-1)$ (degree of $S_{i}$ is $p$ for all $i$, and degree$\binom{(g(x)-1} {p-1}$ = degree($Q_{i}$) = $p-1$).
So, our claim is established.
\end{proof}

\begin{proof}[\bf Proof of Theorem3]
These are some observations :
\begin{enumerate}
\item $P_{18}(x) \equiv 0$ if $g(x) \neq 0$ mod p.

\item $P_{18}(x) \equiv 0$ if $\Sigma_{1\leqslant i\leqslant (9p-3)} a_{i}x_{i}^{(p-1)} \neq 0$ mod $p$ .

\item $P_{18}(x) \equiv 0$ if $\Sigma_{1\leqslant i\leqslant (9p-3)} b_{i}x_{i}^{(p-1)} \neq 0$ mod $p$.

\item $P_{18}(x) \equiv 0$ if $\Sigma_{1\leqslant i\leqslant (9p-3)} c_{i}x_{i}^{(p-1)} \neq 0$ mod $p$.
\end{enumerate}

Define, $m$ = $\Sigma_{x \in Z_{p}^{9p-3}} P_{18}(x)$.

$\\$
After calculating directly we get,
$\\$
$\Sigma_{x \in Z_{p}^{9p-3}} P_{18}(x) \equiv P_{18}(0)+ (p-1)^{p} N^{p}(J) 6! +(p-1)^{8p} N^{8p}(J) 6!$.
$\\$
$m- P_{18}(0)\equiv  (p-1)^{p} N^{p}(J) 6! +(p-1)^{8p} N^{8p}(J) 6!$.
$\\$
Similarly, we can get a set of equations :

$\\$
$m - P_{18}(0) = a_{1} N^{p}(J) + b_{1} N^{8p}(J) $.  
$\\$
$m - P_{17}(0) = a_{2} N^{p}(J) + b_{2} N^{7p}(J) $. $[$ As, $\Sigma_{x \in Z_{p}^{9p-3}} P_{18}(x)$ = $\Sigma_{x \in Z_{p}^{9p-3}} P_{17}(x)$ by  Lemma $3$ $]$ 
$\\$

$\ldots$

$\\$
$m-P_{12}(0) = a_{7} N^{p}(J) + b_{7} N^{2p}(J) $.  $[$ As, $\Sigma_{x \in Z_{p}^{9p-3}} P_{18}(x)$ = $\Sigma_{x \in Z_{p}^{9p-3}} P_{12}(x)$ by  Lemma $3$ $]$ \\
where $a_{i} $, $b_{j}$ are integers. And the values of the $P_{1j}(0)$ are different.
So, the result follows.
\end{proof}

 If we can show that $N^{ip}(J)>0$, $2\leqslant i\leqslant 8$ for a sequence $J$ of $Z_{p}^{3}$ where $|J|=9p-3$ by using the techniques used to prove Theorem3, 
 then it  also proves that $J$ has a zero-sum of length $p$, which proves that $s_{1}(Z_{p}^{3})=9p-3$. And, this is a very good upper
 bound for $s_{1}(Z_{p}^{3})$. This is in the line of the following conjecture,
 \begin{conjecture}
  $s_{1}(Z_{p}^{3})=9p-8$ $[2]$.
 \end{conjecture}
If we can say more about the sum $m=\Sigma_{x \in Z_{p}^{9p-3}} P_{18}(x)$, then it will help
us  to prove our claim.

\section*{Acknowledgements}
I would like to thank Prof. Gautami Bhowmik for her help and her enormous encouragement, without which
this paper would not have been written. I am also thankful to CEMPI-Labex for their financial
support for my Masters2. This article is a part of my master thesis.

\end{document}